\documentclass[12pt]{amsart}
\usepackage{graphicx,amssymb,amsmath,algorithm2e,algorithmic,hyperref}
\title{Generating Minimal Boundary Maps}
\author{Marian Anton, Landon Renzullo}
\address{IMAR, P.O. Box 1-764, 014700 Bucharest, Romania, and CCSU, Dept. of Math Sciences, 1615 Stanley St, New Britain, CT 06050, USA.}
\begin{document}
\maketitle


\newtheorem{thm}{\noindent\bf Theorem:}[section]
\newtheorem{lem}{\noindent\bf Lemma:}[section]
\newtheorem{cor}{\noindent\bf Corrallary:}[section]
\newenvironment{myproof}{\noindent\textbf{Proof:}}{QED\newline}
\newcommand{\N}{\mathbb{N}}
\newcommand{\bigO}{\mathcal{O}}
\newcommand{\pyspace}{\hspace{12pt}}

\section{Abstract}
In this paper we examine the descriptive potential of a combinatorial data structure known as \textbf{Generating Set} in constructing the boundary maps of a simplicial complex. By refining the approach of \cite{Dumas} in generating these maps, we provide algorithms that allow for relations among simplices to be easily accounted for. In this way we explicitly generate faces of a complex only once, even if a face is shared among multiple simplices. The result is a useful interface for constructing complexes with many relations and for extending our algorithms to $\Delta$-complexes. Once we efficiently retrieve the representatives of "living" simplices i.e., of those that have not been related away, the construction of the boundary maps scales well with the number of relations and provides a simpler alternative to JPlex. We finish by noting that the generating data of a complex is equivalent in information to its incidence matrix and we provide efficient algorithms for converting from an incidence matrix to a Generating Set.\newline

\noindent Keywords: topological data analysis; boundary maps; incidence matrix; simplicial complex
\section{Introduction}
We use \cite{hatcher} as a background reference. Let $X = \bigcup_{k = 0}^N X_k$ be a simplicial complex where $X_k$ is the set of $k$-dimensional simplices and $N$ is the maximum dimension of a simplex in $X$. If we denote by $C_k$ the free Abelian group generated by the set $X_k$, the boundary map $\delta_k:C_k\rightarrow C_{k-1}$ is the linear map defined on a $k$-simplex $\alpha=\{v_0,v_1,\dots v_k\}$ by the linear combination
$$\delta_k(\alpha) = \sum_{i = 0}^k (-1)^{i}\alpha_i$$
where $1\le k\le N$ and $\alpha_i=\alpha - \{v_i\}$. For $k>N$, we define $\delta_k$ to be the zero map. We encode this mapping into a matrix $M_k$ by letting each column represent a $k$-simplex $\alpha$ in $X_k$ and each row a $(k-1)$-simplex $\alpha_i$ in $X_{k-1}$ such that the column entry for $(\alpha_i,\alpha)$ is the coefficient $(-1)^i$ of the linear combination $\delta_k(\alpha)$ above and zero otherwise. It is easy to show that $\delta_{k-1}(\delta_{k}(\alpha)) = 0$ for all $k$-simplices $\alpha$, i.e. the element $\delta_k(\alpha)\in C_{k-1}$ is a $(k-1)$-cycle. \newline

Constructing a matrix $M_k$ for each dimension $k$ allows one to compute the homology groups of the chain complex $C=\{C_k\}_{k=0}^N$. The $k$-th homology group is defined as the following quotient: $$H_k(C)=\ker(M_k)/{\rm im}(M_{k+1})$$
where $\ker(M_k)$ is the kernel of $\delta_k$, i.e. all $k$-chains $\alpha$ such that $\delta_k(\alpha) = 0$, and ${\rm im}(M_{k+1})$ is the image of all $(k+1)$-chains under $\delta_{k+1}$.  This quotient is valid because, as we remarked before, the boundary of a boundary is always zero, hence the boundary of any $(k+1)$-simplex must itself have a trivial boundary under $\delta_k$, and therefore ${\rm im}(\delta_{k+1})\subseteq \ker(\delta_k)$. \newline

Once the boundary matrix $M_k$ is constructed, computing homology groups can be then done by the Gaussian row reduction. In particular, by taking the codimension of the image ${\rm im}(\delta_{k+1})$ inside the kernel to retrieve the $k$-th Betti numbers of the complex, i.e. the number of $k$-dimensional "holes" in $C$. As Carlsson points out in \cite{Carlsson} in his seminal paper on topological data analysis (TDA), the Gaussian method is in general too inefficient to be used on large scale complexes. However, if one only desires the Betti numbers of $C$ then alternatively one can use the Smith normal form of our boundary matrices which allows for the computation of all Betti numbers at once and is much more efficient. Algorithms for computing Betti numbers have been studied in depth by others, and in fact much that can be done in the field of topological data analysis relies on efficient computation of Betti numbers over filtered complexes  as in \cite{Dumas, Jager}. Once we have the Betti numbers, another tool known as "barcodes" can be developed as a nice way of representing topological features at multiple time steps or distance thresholds; these are often referred to as "persistent" in our data, as they remain invariant over relatively large intervals of the parameter. See \cite{Ghrist}. \newline

Here however we do not discuss such algorithms in detail, but instead we focus on the generation of the boundary matrix itself. In \cite{Dumas}, some of the difficulties in generating a boundary matrix efficiently are discussed along with a process for trying to generate the boundary map while avoiding the repeat work of the naive algorithm that stems from many simplices sharing the same sub-complex. We first provide a recursive algorithm that can be used to generate any set of independent complexes without repeats. We then expand upon this algorithm to generate any complex with work proportional to that of generating complexes of independent simplices. To do this as generally as possible we introduce the concept of generating sets, which are  a collection of data that describes our complex in a way that is completely combinatorial. One benefit of such a description is the relative ease of describing a type of cellular complex known as the $\Delta$-complex. This is made possible by the construction of a relational map between simplices, which can be deduced from our combinatorial descriptors. We provide algorithms for constructing the relational map, and then in turn the boundary matrix by a method that does not require the explicit generation of any maximal simplex more than once. This gives an algorithmic complexity that improves with the number of relations we have, and also allows us to use a convenient high level description of our complex. We conclude by discussing algorithms for constructing a generating set from the incidence matrix of a complex.

\section{Constructing Boundaries from Generating Sets}
\subsection{Generating Set:}
We define the \textbf{Generating Set} of a complex $C$ to be a pair of sets $G = (S,R)$. The set $$S = \{(n_1,d_1),(n_2,d_2),\dots,(n_k,d_k)\}$$ with $d_i<d_j$ for $i<j$ is a set of tuples where for each tuple $(n_i,d_i)$ there are $n_i$ maximal simplices of dimension $d_i$ in $C$. By \textbf{maximal simplices} we mean simplices that are not the face of another simplex in $C$.  The set $R= \{r_1,r_2,\dots r_m\}$ is a set of relations among faces of maximal simplices where $r_i = (\alpha_1,\dots \alpha_p)$ with all $\alpha_j\in r_i$ of the same dimension. Observe that the number of $k$-simplices generated by $S$ can be calculated by the following function:
$$ j(S,k) = \sum_{(n_i,d_i)\in S}n_i\binom{d_i}{k}.$$

We may also want to calculate $j(S_m,k)$ by the same formula but taken over the set $S_m$ of all pairs $(n_i,d_i)$ with $d_i \leq m$ for a fixed $m$. 

Generating Set can provide a nice interface for constructing certain complexes known as $\Delta$-complexes. Take the complex of the torus as an example. One way to construct it inductively (as JPlex does) is by first adding all necessary vertices, then drawing edges between certain vertices, and finally filling up faces bounded by three edges. This method, while efficient computationally, is tedious, and also does not allow for much simplification of the complex. One valid generating set for the torus would be the following:
$$ S = \{(2,2)\}, R = \{\{(0,1),(4,5)\},\{(0,2),(3,5)\},\{(1,2),(3,4)\}\}$$

As long as we have identified all the faces in $R$ (all sub-simplices are also identified between related simplices), that is sufficient to construct the boundary maps of the $\Delta$-complex of the torus. \newline

To generate the maximal simplices to which we later glue faces to make our complex, we provide the following algorithm:\newline
\begin{algorithm}
	$P(S)$:\newline
	$N = j(S,0)$\newline
	$S^* = \emptyset$\newline
	$c = 0$\newline
	\For{$s_i = (n_i,d_i)\in S$}{
			\For{$1\to n_i$}{
				add \{$c,\dots c+d_i\}$ to $S^*$\newline
				$c = c + d_i + 1$\newline	
						}
			}
	return $S^*$
\end{algorithm}

The algorithm $P(S)$ simply partitions out the vertices into tuples that form all the maximal simplices. 

\subsection{From generating sets to boundary maps}
Let  an ordered set of matrices be given by $D = \{D_0,D_1,\dots,D_N\}$ where $D_i$ is the boundary map from the $i$-chains to the $(i-1)$-chains in the complex $C$. We would like to construct $D$ by using only our Generating Set in a way that is not wasteful. Henceforth, let $G = (S,R)$ be the Generating Set for a complex $C$. Our goal is to generate a list of matrices $D_k$ representing the maps $\delta_k$. Doing this with no relations is straightforward and once we have this, in theory we can add in our relations by simply multiplying each boundary matrix $D_k$ by a permutation matrix $M$ and removing rows and columns of non-representative simplices. Henceforth we may refer to this method as the \textbf{naive} approach to dealing with relations. Also, we may sometimes refer to the \textbf{naive} method of generating faces of a simplex by which we mean computing all boundaries for all faces. Observe, the initial cost of generating the boundaries of all simplices is wasteful given that we know that some (potentially many) of them will disappear after identifications. Besides generating dead simplices, generating the list of permutation matrices is also computationally expensive when done inductively. What we describe is instead a method of constructing boundary maps only for the simplices that are still "alive", while generating relational maps only for simplices that are "dead". We do this by exploiting the patterns predictable from our generation data. 
\subsection{Constructing a relational map}
The first step in our process is to pick a representative for each equivalence class constructed by our relation set, and to keep that representative valid through successive relations. We define the function $Lor:C\rightarrow C/R$ to be the map $\alpha\mapsto \hat\alpha$ where $\hat\alpha$ is the lexicographically lowest simplex to which $\alpha$ is related, i.e. $\hat\alpha=\min\{\beta \;|\; \alpha \equiv \beta\}$. This presents us with one method of constructing the equivalence classes with a valid representative, by simply iterating over all simplices, and placing them in their respective equivalence classes, finding the lowest, and constructing such a mapping. However, this is potentially wasteful in the sense that independent of the number of relations, we will have a runtime of $2^{\bigO(N)}$. In theory this is not necessary as we have our relational data up front and should be able to construct a map only modifying those simplices that appear in our relation. Below we describe an inductive approach to doing so. To make use of our previous notation, we will also consider the maps $Lor_r$ which will assigned the lowest related simplex of $\alpha$ given that the last relation accounted for was $r$. \newline

\begin{thm}
	There exists an algorithm such that given a map $Lor_r$, we can then construct in time $n2^{\bigO(k)}$ the map equivalent to $Lor_s$ for any other remaining relation $s$ where $k$ is the dimension of simplices in $s$ and $n$ is the number of simplices in $s$.
\end{thm}
\begin{proof}	
	Let $H$ be a dictionary or lookup table which stores the mapping of $Lor_r$. Define $RLor(\alpha)$ to be equal to $Lor_r(Lor_r(\alpha))$. Let $s = \{\alpha_1,\dots,\alpha_n\}$ be an unaccounted for relation in $H$. If all $\alpha_i$ in $s$ are not accounted for in $H$ then we simply update $H$ with all $\alpha_i$ mapping to the minimum simplex in $s$. Now suppose that $s$ and $H$ are not disjoint. Here the process is similar in the sense that if the minimum of $s$ is also in $s\cap H$, call it $\beta$, then we simply add all the elements of $s$ to $H$ by mapping $\alpha_i\in s$ to $RLor(\beta)$. Suppose that the minimum of $s$ is not in $s\cap H$. Then for all simplices $\alpha$ in $s\cap H$ we map $RLor(\alpha) = a$ to $\beta$ if $\beta < a$ and vice versa otherwise. Given the recursive definition of $RLor$ we know that anything which had $a$ as its lowest order relation in $H$ will now have the correct relation as well. $RLor$ in tandem with $H$ gives us $Lor_s$. We then must also add all subrelations among $\alpha_i\in s$ induced by $s$. We do this by decomposing all $\alpha_i$ and matching their respective faces of a given dimension giving us new relations. We update all such relations in the same process, without decomposing further. This gives us a final runtime of $n2^{\bigO(k)}$.
\end{proof}

Observe that this is in fact an improvement from the naive approach as in the worst case scenario every simplex in $C$ is in a relation which then requires iteration over all simplices and therefore does the same amount of work as the naive algorithm previously discussed. We now present an algorithm for constructing our list of boundary maps $D$ for the complex $C$ using our generating data. First, we construct a partition of a set of natural numbers from our generating data which will represent our linearly independent simplices without relations. This can be done by computing $P(S)$ from above to give us $\bar{S}$. In the following we also assume access to a function $pos:C\rightarrow\N$ where $pos(x)$ is equal to the relative order of $x$ among simplices of the same dimension in $C$. The arguments of our algorithm are a simplex $\alpha\in \bar{S}$, an index $p$ used in recursion, and the matrix list $D$ which we are updating for the whole complex of $\alpha$.   \newline

\begin{algorithm}
	BCON($\alpha,p$,$D$):\newline
	if $\alpha \neq Lor(\alpha)$: return\newline
	$M = D_{\text{len}(\alpha) - 1}$\newline
	\For{$0\leq i<\text{len}(s)$}{
		$\alpha^* = \alpha - \alpha[i]$\newline
		$a = pos(\alpha),b = pos(Lor(\alpha^*))$\newline
		if $i\neq\text{len}(\alpha)\mod{2}$:
		$\{M[b,a] = M[b,a] + 1\}$\newline
		else:
		$\{M[b,a] = M[b,a] - 1\}$\newline
		if $i > p$: BCON($Lor(\alpha^*)$,$i$,$D$)
							}
\end{algorithm}

\begin{thm}
	Given a map $Lor$ there is an algorithm to compute the minimal boundary matrix for $C$ in time $\bigO(|T|)$ where $T$ is the set of all representatives in $C$.
\end{thm}
	
\begin{proof}
		We claim that BCON$(\alpha,0,D)$ iterated over all simplices $\alpha\in\bar{S}$ is such an algorithm. Observe that BCON, in its first step terminates if the simplex $\alpha$ has already been related away by an explicit relation. Hence, we know that any representative simplex in $\bar{S}$ will continue into the loop. Because we are using $Lor(\alpha^*)$ for all faces of $\alpha$ we know each face gets mapped properly to the representative of its equivalence class. Lastly, the recursive call using the $Lor(\alpha^*)$ guarantees the update all the sub-simplices of $\alpha$ using only the representatives of each equivalence class. 
		Observe that the only way for any simplex to be related away entirely is by an explicit relation or a consequence of an explicit relation of higher order simplex. As we previously showed all simplices not related away explicitly will pass into the loop. Hence we know that BCON not only ignores all non-representative simplices, but also correctly computes boundaries for all representatives.	
\end{proof}
It is worth noting that part of what makes $BCON$ efficient is that it does in fact improve upon the naive approach to boundary generation over a complex, although it may not be immediately obvious. The parameter $p$ that $BCON$ makes use of recursively is actually preventing the generation of repeat faces in $\alpha$. If one wants to pre-generate the simplices in $C$, as they do in \cite{Dumas} for example, this algorithm is straightforward to modify to do so: remove the matrix operations and recursively keep a list of all faces generated from $\alpha$.\newline

We have assumed the existence of a function $pos$ that gives us the relative order of a simplex among all other simplices in $C$ of the same dimension. In practice we can compute this by explicitly generating all sub-complexes of $C$, ordering them lexicographically, and using the position of each simplex to calculate its order. This method suggests an amortized algorithm which makes use of a lookup table storing the location of each simplex. However, we can also calculate the order of a simplex $\alpha$ by solely using $G$ which does not require the generation of the entire complex which up until now we have tried to avoid. For completeness we now illustrate one method of calculating $pos(\alpha)$, with the understanding that there may be better ways to implement this. Notationally we assume $\alpha$ is the set of vertices $\{\alpha_0,\alpha_1,\dots,\alpha_n\}$ with $d_\alpha = |\alpha| - 1$ and $G = (S,R)$

\newpage

\begin{algorithm}
	$Pos(\alpha)$:\newline
	\textbf{if $\alpha_0 \leq j(S_{d_\alpha},0)$}:\newline
	\pyspace\pyspace return $\frac{\alpha_0 - j(S_{d_\alpha - 1},0)}{d_\alpha}$ //$\alpha$ is a maximal simplex, so it is before all non-maximal $d_\alpha$ simps\newline
	\textbf{else}\newline
	find min $x$ where $j(S_x,0)\geq \alpha_0$ \newline
	//finds dimension of maximal simplex $\beta\supset\alpha$ that contains $\alpha$\newline
	$\delta  = \lfloor \frac{\alpha_0 - j(S_{d_\alpha - 1},0)}{x}\rfloor$ //find pos of $\beta$ among x-simps containing $\alpha$\newline
	$a = \delta\binom{x}{d_\alpha} + j(S_x,d_\alpha)$ //number of $d_\alpha$-simps before $\beta$ in $C$\newline
	$b = \sum_{i\leq d_\alpha}\sum_{k}^{\alpha_i - \alpha_{i-1}} \binom{\alpha_i - \alpha_{i-1}}{k}\cdot\binom{|\alpha_i - (\delta+1)\cdot x\cdot j(S_{x-1},0)|}{d_\alpha - k}$\newline
	//number of $d_\alpha$ simps in $\beta$ before $\alpha$\newline
	return $a+b$
\end{algorithm} 

\subsection{Simple examples}

We have already illustrated some of the simplicity that generating sets provide in describing complexes in the example of the torus above. From experimentation with constructing different complexes it is reasonable to conjecture the existence of certain patterns that may simplify the generation of the relational data used in the generating set for certain topological features. Looking back at the $\Delta$-complex for the torus we can see that the relations among boundaries are matched together in reverse lexicographical order (this is not a pattern that necessarily holds in higher dimensional tori). We have yet to look at patterns that seem promising for constructing more general features. Here we conclude with a much simpler example, which at one time JPlex used as an exercise in its tutorial; that being the example of constructing a 7-dimensional hole. In our test code we make use of a function $getBoundary$ that maps a simplex (a list of natural numbers) to its (ordered) set of simplices that make up its boundary. We also make use of $coupleSimps$ which takes two lists of simplices $A,B$ and returns a new list $L = \{(a_i,b_i)|a_i\in A,b_i\in B\}$ which in this instance is our relation set.\newline

\includegraphics[scale = .7]{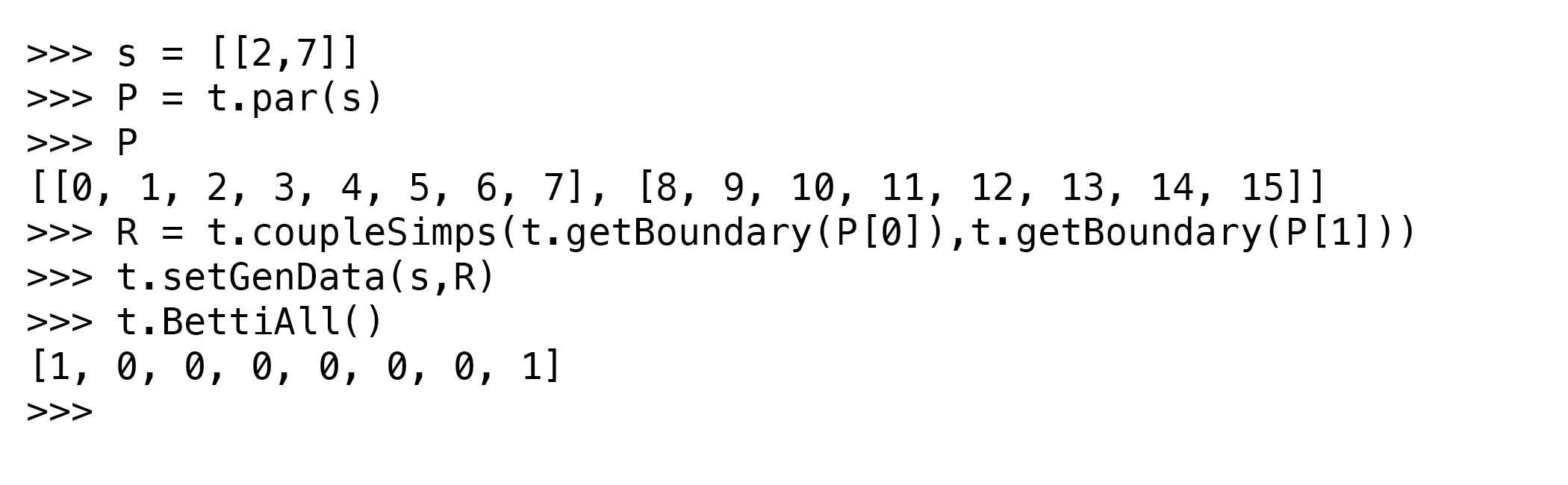} 

This is not the only way to construct this specific example, as we could also use the generating set $S = \{(1,8)\}$ with no relations, and remove the highest order boundary matrix from our list of boundary maps. However, the above example is a good illustration of how easily we can describe a complex through its generating set.  
\section{Converting to Generating Sets}

Consider the representation of a complex $C$ encoded in matrix $M$ in the following way. Similar to the boundary matrix, each column in $M$ represents a simplex in $C$ except all maximal simplices of dimension at least 1 are given a column (as opposed to only simplices of some fixed dimension). The rows of $M$ are given to the vertices of $C$, and each column in turn represents the complex formed by the power set of all vertices with non-zero rows in $M$. We may refer to $M$ as the \textbf{incidence matrix} of complex $C$. This representation of $C$ is equivalent in information to the generating sets we have discussed thus far. 

Note: this version of the incidence matrix is not the same as what is defined in \cite{Awasthi} and possibly elsewhere. The term incidence matrix is sometimes also used in this context to refer to a matrix that encodes the relations among simplices, similar to one of the naive approaches we discussed in section 1 for representing relations. In other contexts it is sometimes used as being nearly equivalent to what we have defined to be the boundary matrix, which is how it is used in \cite{Awasthi}. We will give an algorithm for converting from $M$ to a generating set, with the observation that going from generating sets to a matrix $M$ as described is much more straightforward. \newline

\begin{thm}
	Let $G = (S,R)$ be the generating data for the complex $C$ and let matrix $M_{n\times m}$ be its incidence matrix. Then $S$ can be constructed from $M$ in time $\bigO(mn)$.
\end{thm}
\begin{proof}
	The algorithm is straightforward. For each column of $M$ we count the number of non-zero entries. If this count gives us the value $k$ we then know we have a maximal $(k-1)$-simplex. By keeping a running count of each $(k-1)$-simplex, we get the total number of $k$-simplices for all $1\leq k \leq n$.  
\end{proof}

To encode relational data in a set $R$ is also straightforward in theory. Let a $k$-simplex $\alpha$ and an $\ell$-simplex $\beta$ be in a complex $C$. Suppose that $\alpha,\beta$ share the same $x+1\leq k,\ell$ non-zero rows in their representation in incidence matrix $M$. We then form a relation between $\alpha_{x}$, the $x$-dimensional face of $\alpha$, and $\beta_{x}$, the $x$-dimensional face of $\beta$.

We now present an algorithm for obtaining the relational data of incidence matrix $M$ given that we have already calculated $S$, the simplicial data for complex $C$, from $M$. To do this we assume for the moment access to two functions. The first is  $O:C_k\times \{0\dots \text{dim}(\alpha)\}\rightarrow \{1\dots V\}$ where $V = j(S,0)$ for which $O(\alpha,i)$ is equal to the label in $C$ of the $i$-th vertex in $\alpha$. The second is $O^*:C_k\times C_0/R \rightarrow \{0\dots \text{dim}(\alpha)\} $ where $O^*(\alpha,v)$ is equal to the position of vertex $v$ in $\alpha$ with relations added (we can obtain this from $M$ directly).

\begin{algorithm}
	$Rel(M_{n\times m}):$\newline
	Let $M_c$ be the set of columns in $M$\newline
	Let $R = \emptyset$\newline
	\ForEach{$s,r\in M_c$}{
		Let $I = \{ (O^*(s,v),O^*(r,v))|v\in s\cap r\cap C_0 \}$\newline
		$\alpha_1 = \{O(s,x[0]) |x\in I \}$\newline
		$\alpha_2 = \{O(r,x[1]) |x\in I \}$\newline
		add $(\alpha_1,\alpha_2)$ to $R$
		}
	return $R$
\end{algorithm}

We now provide definitions for the two functions $O$ and $O^*$. They are the following:

\begin{algorithm}
	$O(\alpha,i):$\newline
	$d = dim(\alpha)$\newline
	$\lambda = loc(\alpha)$\newline
	$\bar{S} = S_d \cup \{(\lambda,d)\}$\newline
	return $j(\bar{S},0) + i$
\end{algorithm}

\begin{algorithm}
	$O^*(\alpha,v):$\newline
	$d = dim(\alpha)$\newline
	$\lambda = loc(\alpha)$\newline
	$\bar{S} = S_d \cup \{(\lambda,d)\}$\newline
	$t = j(\bar{S},0), row = 0,c = 0$\newline
	\While{$t + c\neq v$}{
		if $ M_\alpha[row] == 1: c = c + 1$\newline
		$row = row + 1$\newline
		}
		return $c$
	
\end{algorithm}\newpage

Here $\lambda = loc(\alpha)$ is a function where $\alpha$ is the $\lambda$-th simplice of dimension $|\alpha|-1$ in the columns of the incidence matrix $M_c$.\newline

\section{Conclusion}
We have introduced the notion of Generating Sets and have shown how they can be used to construct the boundary matrix of both a simplicial complex as well as its representation as a $\Delta$-complex. We have calculated the complexity of such a construction, and have shown that this is at least an improvement on the naive algorithm of building and gluing inductively. In our approach we have illustrated how given a set of relations we can save work by not repeatedly generating a simplex as both a representative and non-representative of its equivalence class. Finally we conclude by showing that the incidence matrix can be thought of as equivalent to a generating set, and present algorithms to convert from incidence matrix to generating set.


\begin{thebibliography}{9}
	\bibitem{Dumas}J-G. Dumas et. al., \emph{Computing Simplicial Homology Based on Efficient Smith Normal Form Algorithms}, Algebra Geometry and Software Systems, Springer, 2003, 177-206
	\bibitem{Awasthi}V. Awasthi, \emph {A Note on the Computation of Incidence Matrices of Simplicial Complexes}, International Journal of Pure and Applied Mathematics, Vol 8, No 19, 2013, 935-939
	\bibitem{Jager}G. J\"{a}ger, \emph{A New Algorithm for Computing the Smith Normal Form and its Implementation on Parellel Machines}, Mathematisches Seminar, Christian-Albrechts-Universit\"{a}t 
	\bibitem{hatcher} A. Hatcher, \emph{Algebraic Topology}, Cambridge University Press, 2002
	\bibitem{Carlsson} G. Carlsson, \emph{Topology and Data}, Bulletin of the American Mathematical Society, Vol 46, No 2, April 2009, 255-308
	\bibitem{TDASourceCode} \url{https://github.com/fearlesspandas/tda/blob/master/tda2.py}
	\bibitem{Ghrist} R. Ghrist, \emph{The Persistent Topology of Data}, Bulletin of the American Mathematical Society, Vol 45, No 1, January 2009, 61-75 
\end{thebibliography}
\end{document}